\documentclass{article}

\usepackage{a4wide}
\usepackage[isolatin]{inputenc}
\usepackage[portuges,english]{babel}
\usepackage{amssymb}
\usepackage{amsmath}
\usepackage{amsthm}
\usepackage{url}

\newtheorem{theorem}{Teorema}
\newtheorem{corollary}[theorem]{Corol\'{a}rio}
\newtheorem{definition}[theorem]{Defini\c{c}\~{a}o}
\newtheorem{lemma}[theorem]{Lema}
\theoremstyle{definition}
\newtheorem{example}[theorem]{Exemplo}

\newcommand{\Index}[1]{#1\index{#1}}

\begin{document}

\selectlanguage{portuges}

\title{Teorema de Noether no C\'{a}lculo das Varia\c{c}\~{o}es Estoc\'{a}stico}

\author{Adilson C. M. Barros$^1$\\ \texttt{adilson.barros@docente.unicv.edu.cv}
\and Delfim F. M. Torres$^2$\\ \texttt{delfim@ua.pt}}

\date{$^1$Departamento de Ci\^{e}ncia e Tecnologia, Universidade de Cabo Verde\\
Praia, Ilha de Santiago, Cabo Verde\\[0.3cm]
$^2$Centro de Investiga\c{c}\~{a}o e Desenvolvimento em Matem\'{a}tica e Aplica\c{c}\~{o}es\\
Departamento de Matem\'{a}tica, Universidade de Aveiro, 3810-193 Aveiro, Portugal}

\maketitle


\begin{abstract}
Come\c{c}amos por apresentar os problemas cl\'{a}ssicos
do c\'{a}lculo das varia\c{c}\~{o}es determin\'{\i}stico, dando \^{e}nfase \`{a}
condi\c{c}\~{a}o necess\'{a}ria de otimalidade de Euler--Lagrange e
Teorema de Noether. Como exemplos de aplica\c{c}\~{a}o,
obtemos as leis de conserva\c{c}\~{a}o da quantidade de movimento e energia
da mec\^{a}nica, v\'{a}lidas ao longo das extremais de Euler--Lagrange.
Introduzimos depois o c\'{a}lculo das varia\c{c}\~{o}es estoc\'{a}stico e demonstramos
um teorema de Noether estoc\'{a}stico obtido recentemente por Cresson.
Terminamos apontando um problema interessante em aberto.

\medskip

\noindent \textbf{Palavras Chave}: c\'{a}lculo das varia\c{c}\~{o}es, invari\^{a}ncia,
teorema de Noether, leis de conserva\c{c}\~{a}o, c\'{a}lculo das varia\c{c}\~{o}es estoc\'{a}stico,
teorema de Noether estoc\'{a}stico.
\end{abstract}


\selectlanguage{english}

\begin{abstract}
We begin by presenting the classical deterministic problems
of the calculus of variations, with emphasis on the necessary
optimality conditions of Euler--Lagrange and the Noether theorem.
As examples of application, we obtain the conservation laws
of momentum and energy from mechanics, valid along the Euler--Lagrange extremals.
We then introduce the stochastic calculus of variations, proving
a recent stochastic Noether-type theorem obtained by Cresson.
We end by pointing out an interesting open problem.

\medskip


\medskip


\end{abstract}

\selectlanguage{portuges}


\section{Introdu\c{c}\~{a}o}

Um dos primeiros problemas do c\'{a}lculo das varia\c{c}\~{o}es
foi colocado por Galileu em 1630. Consistia
em determinar a trajet\'{o}ria \'{o}tima que
minimizasse o tempo que uma part\'{\i}cula demora a percorrer dois pontos dados,
por a\c{c}\~{a}o exclusiva da gravidade. Algumas d\'{e}cadas mais tarde,
o problema de Galileu viria a ser resolvido pelos irm\~{a}os Bernoulli, Newton, Leibniz,
Euler e Lagrange, quase imediatamente a seguir \`{a} inven\c{c}\~{a}o do c\'{a}lculo diferencial
e integral por Newton e Leibniz.

O problema central do c\'{a}lculo das varia\c{c}\~{o}es consiste
em encontrar uma fun\c{c}\~{a}o $x(\cdot)$ que minimize (ou maximize)
o valor de uma dada funcional integral,
\begin{equation}
\label{eq:prb:var:int}
J[x(\cdot)]=\int_a^b L(t,x(t),\dot{x}(t))dt\longrightarrow \min,
\end{equation}
onde $x : [a,b]\rightarrow \mathbb{R}^n$,
$L:[a,b]\times \mathbb{R}^n \times \mathbb{R}^n \rightarrow \mathbb{R}$
e $\dot{x}=\frac{dx}{dt}$. J\'{a} no s\'{e}culo XIX e princ\'{\i}pio do s\'{e}culo XX,
muitos autores tinham contribu\'{\i}do para a teoria
da solu\c{c}\~{a}o de problemas deste tipo. Entre outros,
sobressaem os nomes de Weierstrass, Bliss e Bolza.
Deste meados do s\'{e}culo XX, o c\'{a}lculo das varia\c{c}\~{o}es alargou-se a v\'{a}rios ramos:
foi crescendo at\'{e} aos dias de hoje, encontrando in\'{u}meras aplica\c{c}\~{o}es pr\'{a}ticas,
na f\'{\i}sica, economia, ci\^{e}ncias dos materiais, engenharia e biologia \cite{SPM:Cristiana:Emmanuel}.
Foi da aplica\c{c}\~{a}o do c\'{a}lculo das varia\c{c}\~{o}es \`{a} f\'{\i}sica que Emmy Noether chegou
ao seu famoso teorema, que estabelece uma rela\c{c}\~{a}o entre a exist\^{e}ncia de
simetrias das funcionais integrais \eqref{eq:prb:var:int} e a exist\^{e}ncia de leis de conserva\c{c}\~{a}o.
Como caso particular do teorema de Noether podemos explicar todas as leis de conserva\c{c}\~{a}o da
mec\^{a}nica. O nosso objetivo principal \'{e} introduzir
o moderno c\'{a}lculo das varia\c{c}\~{o}es estoc\'{a}stico e, em particular,
obter uma formula\c{c}\~{a}o estoc\'{a}stica para o teorema de Noether.

O trabalho encontra-se organizado em seis sec\c{c}\~{o}es.
Come\c{c}amos, na Sec\c{c}\~{a}o~\ref{sec:cv}, por introduzir
o c\'{a}lculo das varia\c{c}\~{o}es cl\'{a}ssico/determin\'{\i}stico,
estudando os conte\'{u}dos mais importantes de um qualquer c\'{a}lculo das varia\c{c}\~{o}es:
formula\c{c}\~{a}o do problema fundamental \eqref{eq:Func},
demonstra\c{c}\~{a}o da equa\c{c}\~{a}o de Euler--Lagrange \eqref{eq:EL}, defini\c{c}\~{a}o de extremal
(Defini\c{c}\~{a}o~\ref{def:extrenal}) e obten\c{c}\~{a}o da condi\c{c}\~{a}o necess\'{a}ria de DuBois--Reymond \eqref{eq:cond:DB:R}.
Na Sec\c{c}\~{a}o~\ref{sec:tcN} demonstramos o teorema cl\'{a}ssico de Noether (Teorema~\ref{TEORMN}).
Como exemplos ilustrativos de leis de conserva\c{c}\~{a}o,
obtemos conserva\c{c}\~{a}o da quantidade de movimento (Exemplo~\ref{ex:cons:mov})
e conserva\c{c}\~{a}o de energia (Exemplo~\ref{ex:cons:energy}). Na Sec\c{c}\~{a}o~\ref{sec:cvEst}
introduzimos, com detalhe, o c\'{a}lculo das varia\c{c}\~{o}es estoc\'{a}stico. O resultado
principal do trabalho, o teorema de Noether estoc\'{a}stico (Teorema~\ref{thm:TN:Est}),
\'{e} demonstrado na Sec\c{c}\~{a}o~\ref{sec:TN:Est}. Terminamos com a Sec\c{c}\~{a}o~\ref{sec:conc}
de conclus\~{a}o, apontando um problema em aberto: a obten\c{c}\~{a}o de um teorema de Noether estoc\'{a}stico
para o caso n\~{a}o aut\'{o}nomo, com mudan\c{c}a da vari\'{a}vel independente.


\section{O C\'{a}lculo das Varia\c{c}\~{o}es}
\label{sec:cv}

O c\'{a}lculo das varia\c{c}\~{o}es \'{e} quase t\~{a}o antigo quanto o pr\'{o}prio C\'{a}lculo,
tendo, os dois assuntos, sido desenvolvidos em paralelo. O c\'{a}lculo variacional \'{e} um alicerce
de muitas teorias em F\'{\i}sica e tem um papel fundamental nessa \'{a}rea bem como na Matem\'{a}tica.
O problema de minimizar funcionais do tipo \eqref{eq:prb:var:int} \'{e}, via de regra,
sujeito a diferentes tipos de restri\c{c}\~{o}es como, por exemplo:
\begin{itemize}
\item condi\c{c}\~{o}es em $x$ nas extremidades do intervalo,
\textrm{i.e.}, $x(a)=A$ e/ou $x(b)=B$;

\item exigir que $g(t,x(t),\dot{x}(t))\equiv 0$ para $t \in [a,b]$,
onde $g:[a,b]\times \mathbb{R}^n \times \mathbb{R}^n\rightarrow \mathbb{R}$ \'{e} dada;

\item exigir que $\int_a^b g(t,x(t),\dot{x}(t))dt=c$, onde
$g:[a,b]\times \mathbb{R}^n \times \mathbb{R}^n\rightarrow \mathbb{R}$
e $c \in \mathbb{R}$.
\end{itemize}
O primeiro tipo de restri\c{c}\~{a}o \'{e} denominado \emph{condi\c{c}\~{a}o de contorno}
e pode ser exigido em ambos ou em apenas um dos extremos do intervalo $[a,b]$.
O segundo tipo de restri\c{c}\~{a}o \'{e} denominado de \emph{restri\c{c}\~{a}o Lagrangiana},
devido \`{a} sua semelhan\c{c}a com as restri\c{c}\~{o}es presentes nos problemas da mec\^{a}nica Lagrangiana.
O terceiro tipo \'{e} denominado de \emph{restri\c{c}\~{a}o isoperim\'{e}trica} (ou integral),
uma vez que os primeiros problemas relacionados com esta restri\c{c}\~{a}o apresentavam
a exig\^{e}ncia dos candidatos $x$ terem todos o mesmo comprimento/per\'{\i}metro.
O intervalo $[a,b]$ n\~{a}o precisa ser necessariamente fixo. Este \'{e} o caso quando uma das
condi\c{c}\~{o}es de contorno \'{e} descrita pela curva de n\'{\i}vel de uma dada fun\c{c}\~{a}o
$\sigma:\mathbb{R}^{1+n}\rightarrow \mathbb{R}$. Temos ent\~{a}o uma restri\c{c}\~{a}o do tipo
\begin{itemize}
\item $\sigma(T,x(T))=0$, $T>a$,
\end{itemize}
que \'{e} denominada \emph{condi\c{c}\~{a}o (de contorno) transversal}.
As restri\c{c}\~{o}es acima podem aparecer de forma combinada,
de modo que um mesmo problema pode estar sujeito a restri\c{c}\~{o}es
de diferentes tipos ou a v\'{a}rias restri\c{c}\~{o}es do mesmo tipo.
Analisamos aqui apenas os problemas com condi\c{c}\~{o}es de contorno.
O problema fundamental consiste na determina\c{c}\~{a}o de uma fun\c{c}\~{a}o
$x(\cdot)\in C^2$ que minimiza uma dada funcional $J[\cdot]$
quando sujeita a duas condi\c{c}\~{o}es de contorno:
\begin{equation}
\label{eq:Func}
\begin{gathered}
J[x(\cdot)]=\int_a^b L\left(t,x(t),\dot{x}(t)\right)dt \longrightarrow \min,\\
x(\cdot) \in  C^2\left([a,b];\mathbb{R}^n\right),\\
x(a)=A, \quad x(b)=B,
\end{gathered}
\end{equation}
onde supomos $a<b$, $A, B \in \mathbb{R}^n$ e $L(\cdot,\cdot,\cdot) \in C^2$
em rela\c{c}\~{a}o a todos os seus argumentos. Para facilidade de apresenta\c{c}\~{a}o,
consideramos de seguida o caso escalar, isto \'{e}, $n=1$.

\begin{theorem}[Condi\c{c}\~{a}o necess\'{a}ria de otimalidade
para o problema \eqref{eq:Func} --- equa\c{c}\~{a}o de Euler--Lagrange]
\label{theo:EL}
Se $x(\cdot)$ \'{e} minimizante do problema \eqref{eq:Func},
ent\~{a}o $x(\cdot)$ satisfaz a equa\c{c}\~{a}o de Euler--Lagrange:
\begin{equation}
\label{eq:EL}
L_x(t,x(t),\dot{x}(t))-\frac{d}{dt} L_{\dot{x}}(t,x(t),\dot{x}(t))=0.
\end{equation}
\end{theorem}

\begin{proof} (seguindo Lagrange)
Seja $x(\cdot)$ um minimizante do problema \eqref{eq:Func}.
Considere-se uma fun\c{c}\~{a}o admiss\'{\i}vel, na vizinhan\c{c}a de $x(\cdot)$, arbitr\'{a}ria.
Tal fun\c{c}\~{a}o pode ser escrita na forma $x(\cdot)+\varepsilon \phi(\cdot)$,
com $\phi(a)=\phi(b)=0$. Por defini\c{c}\~{a}o de minimizante, a fun\c{c}\~{a}o
\begin{equation*}
\Phi(\varepsilon)=J[x(\cdot)+\varepsilon \phi(\cdot)]
= \int_a^b L\left(t,x(t)+\varepsilon \phi(t),\dot{x}(t)+\varepsilon \dot{\phi}(t)\right)dt
\end{equation*}
tem m\'{\i}nimo para $\varepsilon=0$, para qualquer $\phi(\cdot)$. Como a funcional $J$
tem um m\'{\i}nimo em $x(\cdot)$ e $x(a)+\varepsilon\phi(a)=A$ e $x(b)+\varepsilon\phi(b)=B$,
ent\~{a}o $\Phi(\varepsilon)$ tem m\'{\i}nimo no ponto $\varepsilon=0$. Logo,
\begin{equation*}
\Phi{'}(0)=\int_a^b \left(\bar{L}_{x}(t)\phi(t)+\bar{L}_{\dot{x}}(t)\dot{\phi}(t)\right)dt = 0,
\end{equation*}
onde $\bar{L}_{x}(t)=L_{x}(t,x(t),\dot{x}(t))$ e $\bar{L}_{\dot{x}}(t)=L_{\dot{x}}(t,x(t),\dot{x}(t))$.
Integrando por partes, vem que
$$
\int_a^b \bar{L}_{\dot{x}}(t)\dot{\phi}(t)dt
=\bar{L}_{\dot{x}}(t) \phi(t)\mid_a^b
-\int_a^b \frac{d}{dt} \bar{L}_{\dot{x}}(t) \phi(t)
=-\int_a^b \frac{d}{dt} \bar{L}_{\dot{x}}(t) \phi(t)
$$
e, ent\~{a}o,
\begin{equation}
\label{eq:EQLD}
0=\Phi'(0)=\int_a^b \left(\bar{L}_x(t)\phi(t)
+\bar{L}_{\dot{x}}(t)\dot{\phi}(t) \right) dt
=\int_a^b \left(\bar{L}_x(t)-\frac{d}{dt}
\bar{L}_{\dot{x}}(t)\right)\phi(t)dt.
\end{equation}
Utilizamos, agora, o seguinte lema auxiliar,
cuja demonstra\c{c}\~{a}o pode ser encontrada, por exemplo,
em \cite{CD:Frederico} ou em \cite{RelAssociado}.
\begin{lemma}[\Index{Lema fundamental do c\'{a}lculo das varia\c{c}\~{o}es}]
\label{LemaF}
Se $g(\cdot)$ \'{e} cont\'{\i}nua em $[a,b]$ e
\begin{equation*}
\int_a^b g(x)\phi(x)dx=0
\end{equation*}
para todas as fun\c{c}\~{o}es $\phi \in C^2([a,b];\mathbb{R})$ com $\phi(a)=\phi(b)=0$,
ent\~{a}o $g(x) \equiv 0$.
\end{lemma}
\noindent De \eqref{eq:EQLD} resulta do Lema~\ref{LemaF} a conclus\~{a}o pretendida:
$\bar{L}_x(t)-\frac{d}{dt} \bar{L}_{\dot{x}}(t)=0$.
\end{proof}

\begin{definition}
\label{def:extrenal}
As solu\c{c}\~{o}es da equa\c{c}\~{a}o diferencial de Euler--Lagrange \eqref{eq:EL}
(desconsiderando as condi\c{c}\~{o}es de contorno) s\~{a}o denominadas
\emph{fun\c{c}\~{o}es estacion\'{a}rias} ou \emph{extremais},
independentemente do facto de serem ou n\~{a}o solu\c{c}\~{o}es
do problema variacional \eqref{eq:Func}.
\end{definition}

Vamos apresentar um exemplo de aplica\c{c}\~{a}o do Teorema~\ref{theo:EL}.

\begin{example}
Considere-se a funcional integral
\begin{equation}
\label{eq:Prob2}
J[x(\cdot)]=\int_0^\pi \left(\dot{x}^2(t)-kx^2(t)\right) dt
\end{equation}
sob as condi\c{c}\~{o}es de contorno $x(0)=x(\pi)=0$, onde $k$ \'{e} uma constante positiva.
As fun\c{c}\~{o}es admiss\'{\i}veis pertencem \`{a} classe $C^2$.
Se $x(\cdot)$ \'{e} minimizante de \eqref{eq:Prob2}, ent\~{a}o
resulta de \eqref{eq:EL} que
\begin{equation*}
\frac{d}{dt}(2\dot{x}(t))+2kx(t)=0 \Leftrightarrow \ddot{x}(t)+kx(t)=0.
\end{equation*}
Trata-se de uma equa\c{c}\~{a}o diferencial ordin\'{a}ria homog\'{e}nea,
de coeficientes constantes, com solu\c{c}\~{a}o
\begin{equation*}
x(t)=C_1\cos\left(\sqrt{k} t\right)+C_2\sin\left(\sqrt{k} t\right),
\end{equation*}
onde $C_1$ e $C_2$ s\~{a}o constantes de integra\c{c}\~{a}o.
As condi\c{c}\~{o}es de contorno $x(0)=x(\pi)=0$ implicam que
$C_1=0$ e $C_2\sin\left(\sqrt{k}\pi \right)=0$.
Duas situa\c{c}\~{o}es ocorrem, consoante $\sqrt{k}$ \'{e}, ou n\~{a}o, um n\'{u}mero inteiro.
Se $\sqrt{k}\in \mathbb{N}$, ent\~{a}o $\sin\left(\sqrt{k}\pi\right)=0$
e, nesta situa\c{c}\~{a}o, temos uma infinidade de extremais da forma
$x(t)=C_2\sin\left(\sqrt{k}\pi\right)$, $C_2 \in \mathbb{R}$.
Se $\sqrt{k}$ n\~{a}o for um inteiro, ent\~{a}o $C_2=0$
e a \'{u}nica extremal \'{e} $x(t)=0$, $\forall t \in [0,\pi]$.
\end{example}

Uma outra condi\c{c}\~{a}o necess\'{a}ria cl\'{a}ssica
do c\'{a}lculo das varia\c{c}\~{o}es \'{e} a condi\c{c}\~{a}o
necess\'{a}ria de DuBois--Reymond.

\begin{theorem}
\label{CORLDBR}
Uma condi\c{c}\~{a}o necess\'{a}ria para $x(\cdot)$ ser solu\c{c}\~{a}o do problema fundamental
do c\'{a}lculo das varia\c{c}\~{o}es \'{e} dada pela condi\c{c}\~{a}o de DuBois--Reymond:
\begin{equation}
\label{eq:cond:DB:R}
\frac{\partial L}{\partial t}(t,x(t),\dot{x}(t))
=\frac{d}{dt}\left\{L(t,x(t),\dot{x}(t))
-\frac{\partial L}{\partial \dot{x}}(t,x(t),\dot{x}(t))\dot{x}(t)\right\}.
\end{equation}
\end{theorem}

\begin{proof}
Veja-se \cite[Corol\'{a}rio~59]{RelAssociado}.
\end{proof}


\section{Teorema Cl\'{a}ssico de Noether}
\label{sec:tcN}

Emmy Noether, nascida em mar\c{c}o de 1882 na Bav\'{a}ria, Alemanha,
\'{e} considerada a ``m\~{a}e'' da \'{A}lgebra Moderna (Abstrata).
Ap\'{o}s Albert Einstein publicar a sua teoria da Relatividade Geral,
os matem\'{a}ticos ficaram alvoro\c{c}ados e permitiram-se explorar as propriedades desse novo
e revolucion\'{a}rio territ\'{o}rio. A famosa teoria da Relatividade, al\'{e}m de rudimentar
e estranha, apresentava problemas. Emmy Noether deu resposta a importantes quest\~{o}es,
valendo-se da simetria dos problemas variacionais. O teorema de Noether afirma que as
leis de conserva\c{c}\~{a}o s\~{a}o resultado das leis de simetria e constituiu um
grande avan\c{c}o para a \'{e}poca.

Encontrar a solu\c{c}\~{a}o geral da equa\c{c}\~{a}o de Euler--Lagrange \eqref{eq:EL} consiste em determinar
solu\c{c}\~{o}es de uma equa\c{c}\~{a}o diferencial de segunda ordem, geralmente n\~{a}o linear e de dif\'{\i}cil resolu\c{c}\~{a}o.
As leis de conserva\c{c}\~{a}o s\~{a}o especificadas por fun\c{c}\~{o}es $\Phi(t,x(t),\dot{x}(t))$ constantes ao longo
de todas as solu\c{c}\~{o}es $x$ da equa\c{c}\~{a}o de Euler--Lagrange, o que permite baixar
a ordem das equa\c{c}\~{o}es diferenciais dadas pelas condi\c{c}\~{o}es necess\'{a}rias de otimalidade,
simplificando o processo de resolu\c{c}\~{a}o dos problemas do c\'{a}lculo das
varia\c{c}\~{o}es (e controlo \'{o}timo) \cite{ejc}.

Vamos considerar um grupo uni-param\'{e}trico de transforma\c{c}\~{o}es de classe $C^2$, da forma
\begin{equation}
\label{eq:UNIP}
\overline{t}=\Phi(t,x,\varepsilon),
\quad \overline{x}=\Psi(t,x,\varepsilon),
\end{equation}
onde $\Phi$ e $\Psi$ s\~{a}o fun\c{c}\~{o}es anal\'{\i}ticas de
$[a,b]\times \mathbb{R}^n \times ]-\varepsilon,\varepsilon[\rightarrow\mathbb{R}$.
Admitimos que a transforma\c{c}\~{a}o \eqref{eq:UNIP} para $\varepsilon=0$ reduz-se \`{a} identidade:
$\overline{t}=\Phi(t,x,0)=t$ e $\overline{x}=\Psi(t,x,0)=x$.
Numa vizinhan\c{c}a de $\varepsilon=0$,
$\Phi$ e $\Psi$ podem ser expandidas em s\'{e}rie de Taylor:
\begin{equation*}
\begin{gathered}
\overline{t}=t+\varepsilon\frac{\partial \Phi(t,x,0)}{\partial \varepsilon}
+\varepsilon^2\frac{1}{2!}\frac{\partial^2\Phi(t,x,0)}{\partial \varepsilon^2} \cdots
=t+T(t,x)\varepsilon+o(\varepsilon),\\
\overline{x}(\overline{t})=x(t)+\varepsilon\frac{\partial \Psi(t,x,0)}{\partial \varepsilon}
+\varepsilon^2\frac{1}{2!}\frac{\partial^2\Psi(t,x,0)}{\partial \varepsilon^2} \cdots
=x+X(t,x)\varepsilon+o(\varepsilon),
\end{gathered}
\end{equation*}
onde
\begin{equation*}
T(t,x)=\frac{\partial \Phi(t,x,0)}{\partial \varepsilon},
\quad X(t,x)=\frac{\partial \Psi(t,x,0)}{\partial \varepsilon}.
\end{equation*}
Na literatura, as fun\c{c}\~{o}es $T$ e $X$ s\~{a}o designadas por geradores infinitesimais
das transforma\c{c}\~{o}es $\Phi$ e $\Psi$, respetivamente.

\begin{definition}[Invari\^{a}ncia de uma funcional integral \eqref{eq:prb:var:int} do c\'{a}lculo das varia\c{c}\~{o}es]
\label{DEFI}
Dizemos que a funcional integral \eqref{eq:prb:var:int} \'{e} invariante sob as
transforma\c{c}\~{a}o infinitesimais
\begin{equation*}
\begin{cases}
\bar{t}=t+T(t,x)\varepsilon+o(\varepsilon),\\
\bar{x}=x+X(t,x)\varepsilon+o(\varepsilon),
\end{cases}
\end{equation*}
se
\begin{equation}
\label{eq:CONDI2}
\int_{t_a}^{t_b}L(t,x(t),\dot{x}(t))dt
=\int_{\bar{t}(t_a)}^{\bar{t}(t_b)}
L\left(\bar{t}, \bar{x}(\bar{t}), \dot{\bar{x}}(\bar{t})\right)d\bar{t}
\end{equation}
para todo o subintervalo $[t_a,t_b] \subseteq [a,b]$.
\end{definition}

\begin{theorem}[Condi\c{c}\~{a}o necess\'{a}ria de invari\^{a}ncia]
\label{TEORI}
Se a funcional integral \eqref{eq:prb:var:int} \'{e} invariante
no sentido da Defini\c{c}\~{a}o~\ref{DEFI}, ent\~{a}o
\begin{multline}
\label{eq:CMSI}
\frac{\partial L}{\partial t}(t,x(t),\dot{x}(t))T(t,x)
+\frac{\partial L}{\partial x}(t,x(t),\dot{x}(t))X(t,x)\\
+\frac{\partial L}{\partial \dot{x}}(t,x(t),\dot{x}(t))\left(\dot{X}(t,x)
-\dot{x}\dot{T}(t,x)\right)+L(t,x(t),\dot{x}(t))\dot{T}(t,x)=0.
\end{multline}
\end{theorem}

\begin{proof}
A equa\c{c}\~{a}o \eqref{eq:CONDI2} \'{e} valida para todo o subintervalo $[t_a, t_b] \in [a, b]$,
o que nos permite escrever esta equa\c{c}\~{a}o sem o sinal de integral, ou seja, se
\begin{equation*}
\begin{cases}
\bar{t}=t+T(t,x)\varepsilon+o(\varepsilon),\\
\bar{x}=x+X(t,x)\varepsilon+o(\varepsilon),\\
\frac{d\bar{x}}{d\bar{t}}=\frac{\bar{x}+\varepsilon\bar{X}
+o(\varepsilon)}{1+\varepsilon\bar{T}+o(\varepsilon)},
\end{cases}
\end{equation*}
ent\~{a}o
\begin{equation*}
L(t,x(t),\bar{x}(t))=L\left(t+T(t,x)\varepsilon+o(\varepsilon),x+X(t,x)\varepsilon
+o(\varepsilon),\frac{\dot{x}+\varepsilon\dot{X}+o(\varepsilon)}{1+\varepsilon\dot{T}
+o(\varepsilon)}\right)\frac{d\bar{t}}{dt}.
\end{equation*}
Derivando os dois membros da igualdade em ordem a $\varepsilon$,
e fazendo $\varepsilon=0$, vem que
\begin{equation*}
0=\frac{\partial L}{\partial t}(t,x,\dot{x})T
+\frac{\partial L}{\partial x}(t,x,\dot{x})X+\frac{\partial L}{\partial
\bar{x}}(t,x,\dot{x})\left(\dot{X}-\dot{x}\dot{T}\right)+L(t,x,\dot{x})\dot{T}.
\end{equation*}
\end{proof}

\begin{example}
Consideremos a funcional integral
\begin{equation}
\label{PROBN}
J[x(\cdot)]=\int_a^b \dot{x}^2(t)dt.
\end{equation}
Neste caso temos $L(t,x(t),\dot{x}(t))=\dot{x}(t)
\Rightarrow \frac{\partial L}{\partial t}
=\frac{\partial L}{\partial x}=0$ e
$\frac{\partial L}{\partial \dot{x}}=2\dot{x}(t)$. Como
\begin{equation*}
\begin{gathered}
T(t,x)\Rightarrow \dot{T}=\frac{dT}{dt}
=\frac{\partial T}{\partial t}+\dot{x}\frac{\partial T}{\partial x},\\
X(t,x)\Rightarrow \dot{X}=\frac{dX}{dt}
=\frac{\partial X}{\partial t}+\dot{x}\frac{\partial X}{\partial x},
\end{gathered}
\end{equation*}
da equa\c{c}\~{a}o \eqref{eq:CMSI} resulta que
\begin{multline*}
2\dot{x}(t)\left(\frac{\partial X}{\partial t}
+\dot{x}\frac{\partial X}{\partial x}
-\dot{x}\left(\frac{\partial T}{\partial t}
+\dot{x}\frac{\partial T}{\partial x}\right)\right)
+\dot{x}^2\left(\frac{\partial T}{\partial t}
+\dot{x}\frac{\partial T}{\partial x}\right)=0\\
\Leftrightarrow \dot{x}^3\frac{\partial T}{\partial x}
+\dot{x}\left(2\frac{\partial X}{\partial x}
-\frac{\partial T}{\partial t}\right)
+2\dot{x}\frac{\partial X}{\partial t}=0,
\end{multline*}
\begin{gather}
\frac{\partial T}{\partial x}=0, \label{eq:PRON1}\\
2\frac{\partial X}{\partial x}-\frac{\partial T}{\partial t}=0, \label{eq:PRON2}\\
2\frac{\partial X}{\partial t}=0. \label{eq:PRON3}
\end{gather}
De \eqref{eq:PRON1} podemos afirmar que $T$ n\~{a}o depende de $x$.
Ent\~{a}o $T(t,x)=T(t)$ e de \eqref{eq:PRON3} conclu\'{\i}mos que
$X$ n\~{a}o depende de $t$, isto \'{e}, $X(t,x)=X(x)$.
A equa\c{c}\~{a}o \eqref{eq:PRON2} \'{e} satisfeita se
$2\frac{\partial X}{\partial x}=\frac{\partial T}{\partial t}$.
Isto implica que $2\frac{\partial X}{\partial x}=constante
=\frac{\partial T}{\partial t}$, ou seja,
\begin{equation}
\label{eq:PRON4}
X(x)=cx+b_1, \quad T(t)=2ct+b_2,
\end{equation}
onde $c$, $b_1$ e $b_2$ s\~{a}o constantes.
Podemos verificar que os geradores infinitesimais
definidos em \eqref{eq:PRON4} representam as simetrias para que
a funcional \eqref{PROBN} seja invariante no sentido da Defini\c{c}\~{a}o~\ref{DEFI}.
\end{example}

O pr\'{o}ximo teorema \'{e} um dos resultados mais importante da F\'{\i}sica moderna e n\~{a}o s\'{o}:
o teorema de Noether, que foi formulado e demonstrado em 1918 por Emmy Noether, \'{e}
muito mais que um teorema. \'{E} um princ\'{\i}pio geral sobre leis de conserva\c{c}\~{a}o,
com importantes implica\c{c}\~{o}es em v\'{a}rias \'{a}reas da F\'{\i}sica Moderna, na Qu\'{\i}mica, na Economia,
nos problemas Estoc\'{a}sticos, etc.

\begin{theorem}[Teorema de Noether]
\label{TEORMN}
Se a funcional integral \eqref{eq:prb:var:int}
\'{e} invariante no sentido da Defini\c{c}\~{a}o~\ref{DEFI}, ent\~{a}o
\begin{equation}
\label{eq:LGN}
\frac{\partial L}{\partial \dot{x}}(t,x(t),\dot{x}(t))X(t,x(t))
-\left(L(t,x(t),\dot{x}(t))
-\frac{\partial L}{\partial \dot{x}}(t,x(t),\dot{x}(t))\dot{x}(t)\right)T(t,x(t))=const
\end{equation}
\'{e} uma lei de conserva\c{c}\~{a}o.
\end{theorem}

\begin{proof}
Vamos fazer a demonstra\c{c}\~{a}o do teorema de Noether usando
a condi\c{c}\~{a}o de DuBois--Reymond dada pelo Teorema~\ref{CORLDBR}.
A equa\c{c}\~{a}o de Euler--Lagrange \eqref{eq:EL} diz-nos que
\begin{equation*}
\frac{\partial L}{\partial x}(t,x(t),\dot{x}(t))
=\frac{d}{dt}\frac{\partial L}{\partial \dot{x}}(t,x(t),\dot{x}(t)).
\end{equation*}
Podemos ent\~{a}o escrever a condi\c{c}\~{a}o necess\'{a}ria de invari\^{a}ncia \eqref{eq:CMSI} como
\begin{multline*}
\frac{\partial L}{\partial t}(t,x(t),\dot{x}(t))T(t,x(t))
+\frac{d}{dt}\frac{\partial L}{\partial \dot{x}}(t,x(t),\dot{x}(t))X(t,x(t))\\
+\frac{\partial L}{\partial \dot{x}}(t,x(t),\dot{x}(t))\left(\dot{X}(t,x(t))-\dot{x}(t)\dot{T}(t,x(t))\right)
+L(t,x(t),\dot{x}(t))\dot{T}(t,x(t))=0,
\end{multline*}
ou seja,
\begin{multline}
\label{eq:TNOCV}
\frac{d}{dt}
\frac{\partial L}{\partial \dot{x}}(t,x(t),\dot{x}(t))X(t,x(t))
+\frac{\partial L}{\partial \dot{x}}(t,x(t),\dot{x}(t))\dot{X}(t,x(t))\\
+\frac{\partial L}{\partial x}(t,x(t),\dot{x}(t))T(t,x(t))+\dot{T}(t,x(t))\left(L(t,x(t),\dot{x}(t))
-\frac{\partial L}{\partial \dot{x}}(t,x(t),\dot{x}(t))\dot{x}(t)\right)=0.
\end{multline}
Usando a condi\c{c}\~{a}o de DuBois--Reymond (Teorema~\ref{CORLDBR})
podemos escrever \eqref{eq:TNOCV} na forma
\begin{equation*}
\begin{gathered}
\frac{d}{dt}\frac{\partial L}{\partial \dot{x}}\left(t,x(t),\dot{x}(t)\right) X(t,x(t))
+\frac{\partial L}{\partial \dot{x}}(t,x(t),\dot{x}(t))\dot{X}(t,x(t))\\
+\frac{d}{dt}\left\{L(t,x(t),\dot{x}(t))-\frac{\partial L}{\partial \dot{x}}(t,x(t),\dot{x}(t))\right\}T(t,x(t))\\
+\dot{T}(t,x(t))\left(L(t,x(t),\dot{x}(t))-\frac{\partial L}{\partial \dot{x}}(t,x(t),\dot{x}(t))\right)=0\\
\Leftrightarrow \frac{d}{dt}\left\{\frac{\partial L}{\partial \dot{x}}(t,x(t),\dot{x}(t))X(t,x(t))
+\left(L(t,x(t),\dot{x}(t))-\frac{\partial L}{\partial \dot{x}}(t,x(t),\dot{x}(t))\right)T(t,x(t))\right\}=0.
\end{gathered}
\end{equation*}
\end{proof}

A t\'{\i}tulo de exemplo, obtemos as leis de conserva\c{c}\~{a}o
da quantidade de movimento e energia.

\begin{example}[Conserva\c{c}\~{a}o da quantidade de movimento]
\label{ex:cons:mov}
Se a fun\c{c}\~{a}o $L$ n\~{a}o depende de $x$, \textrm{i.e.}, $L=L(t,\dot{x})$,
ocorre conserva\c{c}\~{a}o da quantidade de movimento. Com efeito,
neste caso a equa\c{c}\~{a}o de Euler--Lagrange \eqref{eq:EL} toma a forma
$\frac{d}{dt}L_{\dot{x}}(t,\dot{x}(t))=0$,
o que implica a lei de conserva\c{c}\~{a}o
\begin{equation}
\label{eq:lcm}
L_{\dot{x}}(t,\dot{x}(t))=constante.
\end{equation}
Este facto pode ser obtido por interm\'{e}dio do Teorema~\ref{TEORMN}
j\'{a} que o Lagrangiano $L$ \'{e} invariante sob transla\c{c}\~{o}es em $x$:
podemos fazer $\bar{t}=t$ ($T \equiv 0$) e $\bar{x}=x+\varepsilon$
($X \equiv 1$) e \eqref{eq:LGN} reduz-se a \eqref{eq:lcm}.
\end{example}

\begin{example}[Conserva\c{c}\~{a}o de energia]
\label{ex:cons:energy}
Se a fun\c{c}\~{a}o $L$ n\~{a}o depende de $t$, \textrm{i.e.}, $L=L(x,\dot{x})$,
ent\~{a}o obtemos a lei de conserva\c{c}\~{a}o de energia:
\begin{equation}
\label{eq:ex:energy}
\dot{x}(t) L_{\dot{x}}(x(t), \dot{x}(t))-L(x(t),\dot{x}(t))=constante.
\end{equation}
Com efeito, para toda  a solu\c{c}\~{a}o $x$ da equa\c{c}\~{a}o de Euler--Lagrange, temos:
\begin{equation*}
\begin{split}
\frac{d}{dt}\biggl(\dot{x}(t) &L_{\dot{x}}(x(t),\dot{x}(t))-L\left(x(t),\dot{x}(t)\right)\biggr)\\
&=\frac{d}{dt}\left(\dot{x}(t)L_{\dot{x}}(x(t),\dot{x}(t))\right)-\frac{d}{dt}\left(L(x(t),\dot{x}(t))\right)\\
&=\ddot{x} L_{\dot{x}}+\dot{x} L_{x}-\dot{x}L_x-\ddot{x}L_{\dot{x}}\\
&=0.
\end{split}
\end{equation*}
A lei de conserva\c{c}\~{a}o \eqref{eq:ex:energy} resulta de \eqref{eq:LGN}
devido ao facto do Lagrangiano $L$ ser invariante sob transla\c{c}\~{o}es em $t$:
temos invari\^{a}ncia sob as transforma\c{c}\~{o}es $\bar{t}=t+\varepsilon$ ($T \equiv 1$)
e $\bar{x}=x$ ($X \equiv 0$).
\end{example}


\section{C\'{a}lculo das Varia\c{c}\~{o}es Estoc\'{a}stico}
\label{sec:cvEst}

Nesta sec\c{c}\~{a}o consideramos as \Index{derivadas de Nelson},
propostas por Edward Nelson em 1967 usando um
argumento geom\'{e}trico \cite{CD:ENelson}.
A n\~{a}o diferenciabilidade das trajet\'{o}rias de
\Index{um movimento Browniano},
em processos de Winer, foi usada por Nelson para justificar o facto
de se precisar de um substituto para a derivada cl\'{a}ssica.
Aqui definimos derivadas de Nelson para os Bons Processos de Difus\~{a}o (movimento Browniano)
e damos algumas propriedades das derivadas estoc\'{a}sticas.


\subsection{Derivada Estoc\'{a}stica de Nelson}

Nesta sec\c{c}\~{a}o vamos ver algumas das propriedades das
\Index{derivadas estoc\'{a}sticas}, nomeadamente
a regra da derivada do produto para derivadas de Nelson.
Definimos derivadas estoc\'{a}sticas para funcionais de processos de difus\~{a}o
e especificamos quando um processo \'{e} \Index{Nelson diferenci\'{a}vel}.

Consideremos um espa\c{c}o de probabilidade $(\Omega,{\cal F}, P)$, onde $(\Omega, {\cal F})$
\'{e} um espa\c{c}o mensur\'{a}vel e $P$ uma probabilidade nele definida. O espa\c{c}o amostral $\Omega$
representa o conjunto (suposto n\~{a}o-vazio) de todos os poss\'{\i}veis resultados de uma experi\^{e}ncia
ou fen\'{o}meno aleat\'{o}rio; ${\cal F}$ \'{e} uma $\sigma$-\'{a}lgebra, isto \'{e}, uma classe n\~{a}o-vazia
de subconjuntos de $\Omega$ fechada para o complementar
(se $A \in {\cal F}$, ent\~{a}o o complemento \'{e} definido por $A^c:=\Omega -A \in {\cal F}$)
e para uni\~{o}es cont\'{a}veis (se $A_n \in {\cal F},\ n=1,2,\ldots$, ent\~{a}o $\cup_n A_n \in {\cal F}$).
Os conjuntos $A \in {\cal F}$ s\~{a}o chamados acontecimentos ou conjuntos mensur\'{a}veis.
A probabilidade $P$ \'{e} uma fun\c{c}\~{a}o de ${\cal F}$ em $[0,1]$, normada ($P(\Omega)=1$)
e $\sigma$-aditiva. Seja ${\cal P}_t$ uma filtra\c{c}\~{a}o, isto \'{e}, uma sucess\~{a}o ${\cal P}_t$ da
$\sigma$-\'{a}lgebra de ${\cal F}$, crescente, isto \'{e}, tal que
${\cal P}_t \leq {\cal P}_{t+1}$ e ${\cal F}_t$ \'{e} uma filtra\c{c}\~{a}o decrescente.

\begin{definition}
\label{Def0}
Dizemos que $X$ \'{e} adaptada a ${\cal P}_t$ (adaptada a ${\cal F}_t$)
se, para cada $t$, $X$ \'{e} ${\cal P}_t$ mensur\'{a}vel (${\cal F}_t$ mensur\'{a}vel).
\end{definition}

\begin{definition}
\label{S0}
Seja $X_t(\cdot)$ um processo definido em $I\times\Omega$.
O processo diz-se um S0-processo se: $X_t(\cdot)$ tem um caminho simples cont\'{\i}nuo;
$X_t(\cdot)$ \'{e} ${\cal P}_t$ e ${\cal F}_t$ adaptado para cada $t \in \overline{I}$;
$X_t \in L^2(\Omega)$; e a aplica\c{c}\~{a}o $t \rightarrow X_t$,
de $\overline{I}$ em $L^2(\Omega)$, \'{e} cont\'{\i}nua.
\end{definition}

\begin{definition}[\cite{CD:J.Cresson}]
\label{Def1}
Seja $X_t(\cdot)$ um processo. Dizemos que $X_t(\cdot)$ \'{e} um S1-processo
se ele for um S0-processo e
\begin{gather*}
DX_t :=\lim_{h\rightarrow 0^+}h^{-1}E[X_{t+h}-X_t|{\cal P}_t],\\
D_{*}X_{t} :=\lim_{h\rightarrow 0^+}h^{-1}E[X_t-X_{t-h}|{\cal F}_t],
\end{gather*}
existirem em $L^2(\Omega)$, para $t \in I$, e as aplica\c{c}\~{o}es
$t\rightarrow DX_t$ e $t \rightarrow D_{*}X_{t}$ forem cont\'{\i}nuas de $I$ em $L^2(\Omega)$.
A $DX_t$ e $D_*X_t$ chamamos, respetivamente, derivada de Nelson estoc\'{a}stica \`{a} direita e \`{a} esquerda.
\end{definition}

\begin{definition}
\label{C1}
Denotamos por ${\cal C}^1(I)$ o espa\c{c}o dos S1-processos munido da norma
\begin{equation*}
\|X\|=\sup_{t \in I}\left(\|X(t)\|_{L^2(\omega)}
+\|DX(t)\|_{L^2(\omega)}+\|D_{*}X_{t}\|_{L^2(\omega)}\right).
\end{equation*}
\end{definition}

Vamos agora enunciar um teorema para as derivadas
de Nelson dos bons processos de difus\~{a}o.

\begin{definition}[Defini\c{c}\~{a}o~1.5 de \cite{CD:CressonDarses}]
\label{GD1}
Denotamos por $\Lambda^1$ os processos de difus\~{a}o $X$
que satisfazem as seguintes condi\c{c}\~{o}es:
\begin{enumerate}
\item $X$ \'{e} solu\c{c}\~{a}o da equa\c{c}\~{a}o diferencial estoc\'{a}stica
\begin{equation*}
dX(t)=b(t,X(t))dt+\sigma(t,X(t))dW(t), \ X(0)=X_0,
\end{equation*}
onde $X_0 \in L^2(\Omega), W(\cdot)$ \'{e} movimento ${\cal P}$-Browniano,
$b:\overline{I}\times \mathbb{R}^d\rightarrow \mathbb{R}^d$ e
$\sigma:\overline{I}\times \mathbb{R}^d \rightarrow \mathbb{R}^d\bigotimes\mathbb{R}^d$
s\~{a}o fun\c{c}\~{o}es Borel mensur\'{a}veis satisfazendo as seguintes hip\'{o}teses:
existe uma constante $K$ tal que para todos os $x,y \in \mathbb{R}^d$ temos
\begin{equation*}
\sup_t (|\sigma(t,x)-\sigma(t,y)|+|b(t,x)-b(t,y)|) \leq K|x-y|,
\end{equation*}
\begin{equation*}
\sup_t (|\sigma(t,x)|+|b(t,x)|) \leq K(1+|x|).
\end{equation*}
\item Para todo o $t \in I$, $X(t)$ tem uma fun\c{c}\~{a}o densidade $p_t(x)$ no ponto $x$.
\item Fixando $a_{ij}=(\sigma \sigma^*)_{ij}$,
para qualquer $i \in \{1, \ldots, n \}$ e para qualquer $t_0 \in I$,
\begin{equation*}
\int_{t_0}^1 \int_{\mathbb{R}^d} \left|\partial_j(a_{ij}(t,x)p_t(x))\right|dxdt < +\infty.
\end{equation*}
\item $X(\cdot)$ \'{e} uma difus\~{a}o ${\cal F}$-Browniana.
\end{enumerate}
\end{definition}

\begin{theorem}[\cite{CD:CressonDarses}]
\label{Teror1}
Seja $X \in \Lambda^1$ tal que $dX(t)=b(t,X(t))dt+\sigma(t,X(t))dW(t)$.
Ent\~{a}o $X$ \'{e} difus\~{a}o de Markov com a respetiva filtra\c{c}\~{a}o crescente
$({\cal P}_t)$ e filtra\c{c}\~{a}o decrescente $({\cal F}_t)$.
Al\'{e}m disso, $DX$ e $D_*X$ existem em rela\c{c}\~{a}o a estas filtra\c{c}\~{o}es, sendo dadas por
\begin{equation*}
DX(t) = b(t,X(t)), \quad D_*X(t) = b_*(t,X(t)),
\end{equation*}
onde
\begin{equation*}
b_*^i(t,x)=b^i(t,x)-\frac{\partial_j(a^{ij}(t,x)p_t(x))}{p_t(x)}
\end{equation*}
com $x\rightarrow p_t(x)$ a densidade de $X(t)$ em $x$ e
a conven\c{c}\~{a}o que o termo $\frac{1}{p_t(x)}$ \'{e} zero se $p_t(x)=0$.
\end{theorem}

\begin{definition}
\label{Def3}
Denotamos por ${\cal D}_\mu$ o operador definido por
\begin{equation*}
{\cal D}_\mu=\frac{D+D_*}{2}+i\mu\frac{D-D_*}{2}, \quad \mu=\pm 1.
\end{equation*}
\end{definition}

Vamos definir derivadas estoc\'{a}sticas para as funcionais dos processos
de difus\~{a}o $f(t,X_t)$, onde $X_t$ \'{e} um processo de difus\~{a}o
e $f$ uma fun\c{c}\~{a}o suave.

\begin{definition}
\label{Def4}
Denotamos por $C_b^{1,2}(I\times \mathbb{R}^d)$ o conjunto das fun\c{c}\~{o}es
$f:I \times \mathbb{R}^d\rightarrow \mathbb{R}$, $(t,x)\rightarrow f(t,x)$,
tais que $\partial_t f$, $\nabla f$ e $\partial_{x_i x_j}f$ existem e s\~{a}o limitadas.
\end{definition}

\begin{lemma}[\cite{CD:CressonDarses}]
\label{lema1}
Se $X \in \Lambda^1$ e $f \in C^{1,2}(I\times \mathbb{R}^d)$, ent\~{a}o
\begin{equation*}
Df(t,X(t))=\left[\partial_t f+DX(t) \cdot \nabla f
+\frac{1}{2}a^{ij}\partial_{x_i x_j}f\right](t,X(t)),
\end{equation*}
\begin{equation*}
D_*f(t,X(t))=\left[\partial_t f+D_*X(t) \cdot \nabla f
-\frac{1}{2}a^{ij}\partial_{x_i x_j}f\right](t,X(t)).
\end{equation*}
\end{lemma}

\begin{corollary}[\cite{CD:CressonDarses}]
\label{Cor1}
Se $X \in \Lambda^1$ e $f \in C^{1,2}(I\times \mathbb{R}^d)$, ent\~{a}o
\begin{equation*}
{\cal D}_\mu f(t,X(t))=\left[\partial_t f+{\cal D}_\mu X(t) \cdot \nabla f
+\frac{i\mu}{2}a^{ij}\partial_{x_i x_j}f\right](t,X(t)).
\end{equation*}
\end{corollary}

\begin{corollary}[\cite{CD:CressonDarses}]
\label{Cor2}
Se $X \in \Lambda^1$ com coeficiente de difus\~{a}o constante
$\sigma$ e $f \in C^{1,2}(I\times \mathbb{R}^d)$, ent\~{a}o
\begin{equation*}
{\cal D}_\mu f(t,X(t))
=\left[\partial_t f+{\cal D}_\mu X(t) \cdot \nabla f
+\frac{i \mu {\sigma}^2}{2}\Delta f\right](t,X(t)).
\end{equation*}
\end{corollary}

O teorema que se segue d\'{a}-nos a f\'{o}rmula para o c\'{a}lculo da derivada
do produto estoc\'{a}stico de Nelson, an\'{a}loga \`{a} regra usual de Leibniz
$(fg)'=f'g+fg'$ da derivada do produto.

\begin{theorem}[\cite{CD:CressonDarses}]
\label{TEORDP}
Se $X,Y \in {\cal C}^1(I)$, ent\~{a}o
\begin{equation*}
\frac{d}{dt}E[X(t)Y(t)]=E\left[DX(t) \cdot Y(t)+X(t) \cdot D_*Y(t)\right].
\end{equation*}
\end{theorem}

\begin{lemma}[\cite{CD:CressonDarses}]
\label{lema2}
Se $X,Y \in {\cal C}^1(I)$, ent\~{a}o
\begin{gather*}
\frac{d}{dt}E[X(t)Y(t)]
=E\left[Re({\cal D}X(t)) \cdot Y(t)+X(t) \cdot Re({\cal D}Y(t))\right],\\
E\left[Im({\cal D}X(t)) \cdot Y(t)\right]
=E\left[X(t) \cdot Im({\cal D}Y(t))\right].
\end{gather*}
\end{lemma}

Vamos agora definir quando um processo \'{e} Nelson diferenci\'{a}vel.

\begin{definition}
\label{Def5}
Um processo $X$ \'{e} dito de Nelson diferenci\'{a}vel se $DX=D_*X$.
Escrevemos ent\~{a}o que $X \in {\cal N}^1(I)$.
\end{definition}

\begin{corollary}[\cite{CD:CressonDarses}]
\label{CORL3}
Sejam $X,Y \in {\cal C}_{\mathbb{C}}^1(I)$. Se $X$ \'{e} Nelson diferenci\'{a}vel, ent\~{a}o
\begin{equation}
\label{eq:CORL}
E\left[{\cal D}_\mu X(t) \cdot Y(t)+X(t) \cdot {\cal D}_\mu Y(t)\right]
=\frac{d}{dt}E[X(t)Y(t)].
\end{equation}
\end{corollary}

Seguindo a abordagem feita em \cite{CD:CressonDarses},
vamos associar uma equa\c{c}\~{a}o de Euler--Lagrange estoc\'{a}stica
a uma funcional estoc\'{a}stica, mostrando que existe uma analogia
com o Princ\'{\i}pio da A\c{c}\~{a}o M\'{\i}nima. Para isso come\c{c}amos
por introduzir os conceitos de funcional estoc\'{a}stica e processo $L$-adaptado.


\subsection{Funcional e Processo $L$-adaptado}

Nesta sec\c{c}\~{a}o denotamos por $I$ um dado intervalo aberto $(a,b)$, $a < b$.

\begin{definition}
\label{FUNAD}
Um Lagrangiano admiss\'{\i}vel \'{e} uma fun\c{c}\~{a}o $L$ tal que
\begin{enumerate}
\item A fun\c{c}\~{a}o $L(x,v,t)$ \'{e} definida em
$\mathbb{R}^d \times \mathbb{C}^d \times \mathbb{R}$, holomorfa na segunda vari\'{a}vel.
\item $L$ \'{e} aut\'{o}nomo, isto \'{e}, $L$ n\~{a}o depende do tempo: $L(x,v,t) = L(x,v)$.
\end{enumerate}
\end{definition}

\begin{definition}
\label{FUNCE}
Seja $L$ um Lagrangiano admiss\'{\i}vel e
\begin{equation*}
\Xi = \left\{X \in  {\cal C}^1(I), E\left[\int_I |L(X(t),{\cal D}_\mu X(t))|dt \right]
< \infty \right\}.
\end{equation*}
A funcional associada a $L$ \'{e} definida por
\begin{equation}
\label{eq:CVESC1}
F_I:
\begin{cases}
\Xi\rightarrow  \mathbb{C}\\
\displaystyle X\mapsto E \left[\int_I L(X(t),{\cal D}_\mu X(t))dt\right].
\end{cases}
\end{equation}
\end{definition}

\begin{definition}
\label{FUNCE1}
Seja  $L$ um  Lagrangiano.
Um processo $X \in {\cal C}^1(I)$ \'{e} chamado de $L$-adaptado se:
\begin{enumerate}
\item para todo o $t \in I$, $\partial_x L(X(t),{\cal D}_\mu X(t))$ \'{e} ${\cal P}_t$ e ${\cal F}_t$
mensur\'{a}vel e  $\partial_x L(X(t),{\cal D}_\mu X(t)) \in L^2(\Omega)$;
\item $\partial_v L(X(t),{\cal D}_\mu X(t)) \in {\cal C}^1(I)$.
\end{enumerate}
\end{definition}


\subsection{Espa\c{c}o Variacional}

O C\'{a}lculo das Varia\c{c}\~{o}es considera o comportamento das varia\c{c}\~{o}es
no \^{a}mbito do espa\c{c}o funcional subjacente. Um cuidado especial deve ser tomado no caso
estoc\'{a}stico, para definir qual \'{e} a classe de varia\c{c}\~{o}es que estamos a considerar.
Usamos a seguinte defini\c{c}\~{a}o:

\begin{definition}
\label{FUNCE2}
Seja $\Gamma$ um subespa\c{c}o de ${\cal C}^1(I)$.
A $\Gamma$-varia\c{c}\~{a}o de $X$ \'{e} um processo estoc\'{a}stico
da forma $X+Z$, onde $Z \in \Gamma$. Por outro lado,
\begin{equation*}
\Gamma_\Xi=\{Z \in \Gamma, \forall X \in \Xi, Z+X \in \Xi\},
\end{equation*}
onde $\Gamma_\Xi$ denota um subespa\c{c}o de $\Gamma$.
\end{definition}

Vamos considerar dois subespa\c{c}os variacionais: ${\cal N}^1(I)$ e  ${\cal C}^1(I)$.


\subsection{\Index{Funcional $\Gamma$-Diferenci\'{a}vel e Processo $\Gamma$-Cr\'{\i}tico}}

Vamos definir fun\c{c}\~{a}o diferenci\'{a}vel.
No que se segue $\Gamma$ \'{e} um subespa\c{c}o de ${\cal C}^1(I)$.

\begin{definition}
\label{FUNCE3}
Seja $L$ um Lagrangiano admiss\'{\i}vel. A funcional $F_I$ diz-se
$\Gamma$-diferenci\'{a}vel em $X \in \Xi\cap L$ se
\begin{equation*}
F_I(X+Z)-F_I(X)=dF_I(X,Z)+R(X,Z)
\end{equation*}
para todo o $Z \in \Gamma_\Xi$, onde $dF_I(X,Z)$ \'{e} uma fun\c{c}\~{a}o linear
e $R(X,Z)=o(\parallel Z \parallel)$.
\end{definition}

A nossa pr\'{o}xima defini\c{c}\~{a}o \'{e} a de processo cr\'{\i}tico estoc\'{a}stico.

\begin{definition}
\label{FUNCE4}
Um processo $\Gamma$-cr\'{\i}tico para a funcional $F_I$
\'{e} um processo estoc\'{a}stico $X \in \Xi\cap L$ tal que
$dF_I(X,Z)=0$ para todos os $Z \in \Gamma_\Xi$ com
$Z(a)=Z(b)=0$.
\end{definition}

O Lema~\ref{lem:33} considera o caso em que $\Gamma={\cal C}^1(I)$
enquanto o Lema~\ref{lem:33} considera $\Gamma={\cal N}^1(I)$.

\begin{lemma}[Cap.~7 de \cite{CD:CressonDarses}]
\label{lem:33}
Seja $L$ um Lagrangiano admiss\'{\i}vel com todas as segundas derivadas limitadas.
A funcional $F_I$ definida em \eqref{eq:CVESC1} \'{e} ${\cal C}^1(I)$-diferenci\'{a}vel
para todo o $X \in \Xi \cap L$ e ${\cal C}^1(I)_\Xi={\cal C}^1(I)$. Para todos os
$Z \in {\cal C}^1(I)$, o diferencial de $F_I$ \'{e} dado por
\begin{multline*}
dF_I(X,Z)=E\left[\int_a^b\left[\frac{\partial L}{\partial x}(X(u),{\cal D}_\mu X(u))
-{\cal D}_{-\mu}\left(\frac{\partial L}{\partial x}(X(u),{\cal D}_\mu
X(u))\right)\right]Z(u)du \right]\\
+g(Z,\partial_v L)(b)-g(Z,\partial_v L)(a),
\end{multline*}
onde $g(Z,\partial_v L)(s)=E\left[Z(u)\partial_vL(X(u),{\cal D}_\mu X(u))\right]$.
\end{lemma}

\begin{lemma}[Cap.~7 de \cite{CD:CressonDarses}]
\label{lem:34}
Seja $L$ um Lagrangiano admiss\'{\i}vel com todas as segundas derivadas limitadas.
A funcional $F_I$ definida em \eqref{eq:CVESC1} \'{e} ${\cal N}^1(I)$-diferenci\'{a}vel
para todo o $X \in \Xi \cap {\cal L}$ e ${\cal N}^1(I)_\Xi={\cal N}^1(I)$.
Para todos os $Z \in {\cal N}^1(I)$, o diferencial \'{e} dado por
\begin{multline*}
dF_I(X,Z)=E\left[\int_a^b\left[\frac{\partial L}{\partial x}(X(u),{\cal D}_\mu X(u))
-{\cal D}_{\mu}\left(\frac{\partial L}{\partial x}(X(u),{\cal D}_\mu X(u))\right)\right]Z(u)du \right]\\
+g(Z,\partial_v L)(b)-g(Z,\partial_v L)(a),
\end{multline*}
onde $g(Z,\partial_v L)(s)=E[Z(u)\partial_vL(X(u),{\cal D}_\mu X(u))]$.
\end{lemma}


\subsection{Equa\c{c}\~{a}o de Euler--Lagrange Estoc\'{a}stica}

Nesta sec\c{c}\~{a}o vamos determinar a equa\c{c}\~{a}o de Euler--Lagrange estoc\'{a}stica.
Tamb\'{e}m aqui, consideramos dois casos para o subespa\c{c}o variacional:
${\cal N}^1(I)$ e ${\cal C}^1(I)$.


\subsubsection{Caso ${\cal C}^1(I)$}

\begin{theorem}
Seja $L$ um Lagrangiano admiss\'{\i}vel com todas as segundas derivadas limitadas.
Uma condi\c{c}\~{a}o necess\'{a}ria e suficiente para que o processo
$\Xi \in {\cal L}\cap {\cal C}^3(I)$ seja um processo
${\cal C}^1(I)$-cr\'{\i}tico associado \`{a} funcional $F_I$ \'{e} dada por
\begin{equation}
\label{eq:EELEST}
\frac{\partial L}{\partial x}(X(t),{\cal D}_\mu X(t))
-{\cal D}_{-\mu}\left[\frac{\partial L}{\partial v}(X(t),{\cal D}_\mu X(t))\right]=0.
\end{equation}
\end{theorem}

\begin{proof}
Sem perda de generalidade, consideramos $I=(0,1)$.
Seja $X \in {\cal C}^3(I)$ uma solu\c{c}\~{a}o de
\begin{equation*}
\frac{\partial L}{\partial x}(X(t),{\cal D}_\mu X(t))
-{\cal D}_{-\mu}\left[\frac{\partial L}{\partial v}(X(t),{\cal D}_\mu X(t))\right]=0.
\end{equation*}
Ent\~{a}o, $X$ \'{e} ${\cal C}^1(I)$-cr\'{\i}tico associado \`{a} funcional $F_I$.
Se $X$ \'{e} ${\cal C}^1(I)$-cr\'{\i}tico associado \`{a} funcional $F_I$, isto \'{e}, $dF_I(X,Z)=0$, temos
\begin{equation*}
Re(dF_I(X,Z))=Im(dF_I(X,Z))=0.
\end{equation*}
Vamos definir
\begin{equation*}
Z_n^{(1)}(u)=\phi_n^{(1)}(u).Re\left(\frac{\partial L}{\partial x}(X(t),{\cal D}_\mu X(t))
-{\cal D}_{-\mu}\left[\frac{\partial L}{\partial v}(X(t),{\cal D}_\mu X(t))\right]\right)
\end{equation*}
e
\begin{equation*}
Z_n^{(2)}(u)=\phi_n^{(2)}(u).Im\left(\frac{\partial L}{\partial x}(X(t),{\cal D}_\mu X(t))
-{\cal D}_{-\mu}\left[\frac{\partial L}{\partial v}(X(t),{\cal D}_\mu X(t))\right]\right),
\end{equation*}
onde $\left(\phi_n^{(i)}\right)_{n \in N}$ \'{e} uma sequ\^{e}ncia de $C^\infty([0,1]\rightarrow \mathbb{R}^+)$
determin\'{\i}stica em $[0,1]$, isto \'{e}, para todo o $n \in N$, $\phi_n(0)=\phi_n(1)=0$ e $\phi_n=1$
em $[\alpha_n,\beta_n]$ com $0<\alpha_n$, $\beta_n<1$, $\lim_{n\rightarrow \infty}\alpha_n=0$
e $\lim_{n\rightarrow \infty}\beta_n=1$. Assim, para todo o $n \in N$,
\begin{equation*}
Re(dF_I(X,Z_n^{(1)}))=E\left[\int_0^1
\phi_n(u)Re^2\left(\frac{\partial L}{\partial x}(X(t),{\cal D}_\mu X(t))
-{\cal D}_{-\mu}\left[\frac{\partial L}{\partial v}(X(t),{\cal D}_\mu X(t))\right]\right)du\right]=0,
\end{equation*}
\begin{equation*}
E\left[\int_0^1 Re^2\left(\frac{\partial L}{\partial x}(X(t),{\cal D}_\mu X(t))
-{\cal D}_{-\mu}\left[\frac{\partial L}{\partial v}(X(t),{\cal D}_\mu X(t))\right]\right)du\right]=0.
\end{equation*}
Usando o mesmo argumento,
\begin{equation*}
E\left[\int_0^1 Im^2\left(\frac{\partial L}{\partial x}(X(t),{\cal D}_\mu X(t))
-{\cal D}_{-\mu}\left[\frac{\partial L}{\partial v}(X(t),{\cal D}_\mu X(t))\right]\right)du\right]=0.
\end{equation*}
Portanto, para quase todos os $t \in [0,1]$ e quase todos $w \in \Omega$,
\begin{equation*}
\frac{\partial L}{\partial x}(X(t),{\cal D}_\mu X(t))
-{\cal D}_{-\mu}\left[\frac{\partial L}{\partial v}(X(t),{\cal D}_\mu X(t))\right]=0.
\end{equation*}
\end{proof}

A equa\c{c}\~{a}o \eqref{eq:EELEST} \'{e} chamada
de equa\c{c}\~{a}o de Euler--Lagrange estoc\'{a}stica.


\subsubsection{Caso ${\cal N}^1(I)$}

A demonstra\c{c}\~{a}o do pr\'{o}xima lema pode ser encontrada
no Cap.~7 de \cite{CD:CressonDarses}.

\begin{lemma}
Seja $L$ um  Langrangiano admiss\'{\i}vel com as segundas derivadas limitadas.
Se $X$ \'{e} solu\c{c}\~{a}o da equa\c{c}\~{a}o de Euler--Lagrange estoc\'{a}stica
\begin{equation*}
\frac{\partial L}{\partial x}(X(t),{\cal D}_\mu X(t))
-{\cal D}_{\mu}\left[\frac{\partial L}{\partial v}(X(t),{\cal D}_\mu X(t))\right]=0,
\end{equation*}
ent\~{a}o $X$ \'{e} um processo ${\cal N}^1(I)$-cr\'{\i}tico para a funcional $F_I$ associada a $L$.
\end{lemma}


\section{Teorema de Noether Estoc\'{a}stico}
\label{sec:TN:Est}

Nesta sec\c{c}\~{a}o vamos demonstrar um Teorema de Noether
no contexto estoc\'{a}stico para problemas aut\'{o}nomos.
Para isso vamos definir, seguindo Cresson \cite{CD:J.Cresson},
vetor tangente a um processo estoc\'{a}stico, suspens\~{a}o estoc\'{a}stica de uma fam\'{\i}lia param\'{e}trica
de difeomorfismos, invari\^{a}ncia e primeiro integral estoc\'{a}stico.

Seja $X \in {\cal C}^1(I)$ um processo estoc\'{a}stico. Definimos \emph{vetor tangente}
a um processo estoc\'{a}stico de modo an\'{a}logo a vetor tangente de $X$ no ponto $t$.

\begin{definition}
\label{DefNE1}
Seja $X \in {\cal C}^1(I)$, $I\subset \mathbb{R}$.
O vetor tangente a $X$ no ponto $t$ \'{e} dado por ${\cal D}X(t)$.
\end{definition}

Vamos agora introduzir a no\c{c}\~{a}o de suspens\~{a}o estoc\'{a}stica
de uma fam\'{\i}lia param\'{e}trica de difeomorfismos.

\begin{definition}
\label{DefNE2}
Seja $\phi:\mathbb{R}^d\rightarrow \mathbb{R}^d$ um difeomorfismo.
A suspens\~{a}o estoc\'{a}stica de $\phi$ \'{e} a aplica\c{c}\~{a}o $\Phi:P\rightarrow P$
definida por
\begin{equation*}
\forall X \in P, \, \Phi(X)_t(w)=\phi(X_t(w)).
\end{equation*}
\end{definition}

\begin{definition}
\label{DefNE3}
Um grupo uni-param\'{e}trico de transforma\c{c}\~{o}es
$\Phi_s:\Upsilon\rightarrow \Upsilon,\ s \in \mathbb{R}$,
onde $\Upsilon \subset P$, \'{e} chamado de $\phi$-grupo de suspens\~{a}o
agindo em $\Upsilon$ se existir um grupo uni-param\'{e}trico de difeomorfismos
$\phi_s:\mathbb{R}^d\rightarrow \mathbb{R}^d$, $s \in \mathbb{R}$,
de tal forma que para todo o $s \in \mathbb{R}$ temos:
\begin{enumerate}
\item $\Phi_s$ \'{e} uma suspens\~{a}o estoc\'{a}stica de $\phi_s$;
\item  $X \in \Upsilon, \ \phi_s(X) \in \Upsilon$.
\end{enumerate}
\end{definition}

\begin{lemma}[\cite{CD:J.Cresson}]
\label{lemaEN1}
Seja $\Phi=(\phi_s)_{s \in \mathbb{R}}$ uma suspens\~{a}o estoc\'{a}stica
de um grupo uni-param\'{e}trico de difeomorfismos. Ent\~{a}o, para todo o $X \in \Lambda$,
temos para todo o $t \in I$ e todo o $s \in R$ que
\begin{enumerate}
\item a aplica\c{c}\~{a}o $s\mapsto {\cal D}_\mu\Phi_sX(t) \in C^1(\mathbb{R})$;
\item $\frac{\partial}{\partial s}[{\cal D}_\mu(\phi_s(X))]
={\cal D}_\mu \left[\frac{\partial \phi_s(X)}{\partial s}\right]$.
\end{enumerate}
\end{lemma}

Seja $X \in \cal{C}^1(I)$ e $\phi : \mathbb{R}^d \rightarrow \mathbb{R}^d$
um difeomorfismo. A imagem de $X$ sob a suspens\~{a}o estoc\'{a}stica de $\phi$,
denotada por $\Phi$, induz de modo natural uma aplica\c{c}\~{a}o para os vetores
tangentes, denotada por $\Phi_*$ e chamada de aplica\c{c}\~{a}o linear tangente,
definida como na geometria diferencial cl\'{a}ssica:

\begin{definition}
Seja $\Phi$ uma suspens\~{a}o estoc\'{a}stica de um difeomorfismo $\phi$
tal que a sua $k$-\'{e}sima componente $\phi^{(k)}$ pertence a $\mathcal{T}$.
A aplica\c{c}\~{a}o linear tangente associada a $\Phi$, e denotada por $\Phi_*$,
\'{e} definida para todo o $X \in \cal{C}^1(I)$ por
$\Phi_*(X) = T(\Phi(X)) = (\Phi(X),{\cal D}(\Phi(X)))$.
\end{definition}

Obtemos ent\~{a}o a seguinte no\c{c}\~{a}o de invari\^{a}ncia sob a a\c{c}\~{a}o
de um grupo uni-param\'{e}trico de difeomorfismos.

\begin{definition}
\label{DefNE4}
Seja $\Phi=\{\phi_s\}_{s \in \mathbb{R}}$ um grupo uni-param\'{e}trico
de difeomorfismos e seja $L:{\cal C}^1(I)\rightarrow {\cal C}_{\mathbb{C}}^1(I)$.
A funcional $L$ \'{e} invariante sob $\Phi$ se
$L(\phi_*X)=L(X)$ para todos o $\phi \in \Phi$.
\end{definition}

Como consequ\^{e}ncia da Defini\c{c}\~{a}o~\ref{DefNE4}, se $L$ \'{e} invariante temos
que $L\left(\phi_*X; D(\phi_*X)\right)=L(X, DX)$
para todo o $s \in \mathbb{R}$ e $X \in {\cal C}^1(I)$.
Estamos em condi\c{c}\~{o}es de enunciar e demonstrar
o teorema de Noether no contexto estoc\'{a}stico.

\begin{theorem}[\Index{Teorema de Noether Estoc\'{a}stico}]
\label{thm:TN:Est}
Seja $L$ um Lagrangiano admiss\'{\i}vel com todas as segundas derivadas limitadas
e invariante sob um grupo uni-param\'{e}trico de difeomorfismos
$\Phi=\{\phi_s\}_{s \in \mathbb{R}}$. Seja $F_I$ a funcional associada a $L$
e definida por \eqref{eq:CVESC1} em $\Xi$. Seja $X \in \Xi \cap L$
de classe ${\cal C}^1(I)$ um ponto estacion\'{a}rio de $F_I$. Ent\~{a}o,
\begin{equation*}
\frac{d}{dt}E \left[\partial_v L
\cdot \left.\frac{\partial Y}{\partial s}\right|_{s=0}\right]=0,
\end{equation*}
onde $Y_s=\Phi_s(X)$.
\end{theorem}

\begin{proof}
Seja $Y(s,t)=\phi_s X(t)$, $s \in \mathbb{R}$ e $t \in [a, b]$.
Se $L$ \'{e} invariante sob $\Phi=\{\phi_s\}_{s \in \mathbb{R}}$, ent\~{a}o
\begin{equation*}
\frac{\partial}{\partial s}L(Y(s,t), {\cal D}_\mu Y(s,t))=0
\end{equation*}
com $Y(\cdot,t)$ e ${\cal D}_\mu Y(\cdot,t) \in C^1(\mathbb{R})$ $\forall$ $t \in [a,b]$.
Temos, por conseguinte,
\begin{equation*}
\partial_x L \cdot \frac{\partial Y}{\partial s}
+\partial_v L \cdot \frac{\partial {\cal D}_\mu Y}{\partial s}=0
\end{equation*}
que \'{e} equivalente a
\begin{equation*}
\partial_x L \cdot \frac{\partial Y}{\partial s}
+\partial_v L \cdot {\cal D}_\mu\left(\frac{\partial  Y}{\partial s}\right)=0
\end{equation*}
com $X=Y|_{s=0}$ um processo estacion\'{a}rio para $F_I$. Resulta ent\~{a}o que
$\partial_x L={\cal D}_{\mu} \partial_v L$. Como consequ\^{e}ncia,
\begin{equation*}
[{\cal D}_\mu \partial_v L] \cdot \frac{\partial Y}{\partial s}
+\partial_v L \cdot {\cal D}_\mu\left(\frac{\partial  Y}{\partial s}\right)=0
\end{equation*}
e
\begin{equation*}
E\left[[{\cal D}_\mu \partial_v L] \cdot \frac{\partial Y}{\partial s}
+\partial_v L \cdot {\cal D}_\mu\left(\frac{\partial  Y}{\partial s}\right)\right]=0.
\end{equation*}
Usando a regra do produto \eqref{eq:CORL},
\begin{equation*}
\frac{d}{dt}E \left[\partial_v L\cdot \left.\frac{\partial Y}{\partial s}\right|_{s=0}\right]=0.
\end{equation*}
\end{proof}


\section{Conclus\~{a}o}
\label{sec:conc}

Ao longo deste trabalho estud\'{a}mos problemas estoc\'{a}sticos
do c\'{a}lculo das varia\c{c}\~{o}es e, em particular, obtivemos
uma formula\c{c}\~{a}o estoc\'{a}stica do Teorema de Noether.
Generaliza\c{c}\~{o}es no contexto do controlo \'{o}timo
podem ser encontradas na disserta\c{c}\~{a}o de mestrado
do primeiro autor \cite{MSc:AB}.

O nosso estudo dos problemas estoc\'{a}sticos do c\'{a}lculo das varia\c{c}\~{o}es
e do Teorema de Noether estoc\'{a}stico, foi realizado
tendo em conta a abordagem proposta por Cresson e
Darses em \cite{CD:CressonDarses}.
Outras abordagens s\~{a}o poss\'{\i}veis, como seja a de
Zambrini \cite{CD:Zamb1}, que em 1980 apresentou
duas extens\~{o}es do teorema fundamental do c\'{a}lculo das varia\c{c}\~{o}es estoc\'{a}stico,
importantes para problemas da f\'{\i}sica que envolvem restri\c{c}\~{o}es \cite{CD:Zamb3,CD:Zamb2}.
L\'azaro-Cam\'\i\  e Ortega, utilizando ferramentas
da an\'{a}lise global estoc\'{a}stica introduzidas por Meyer e Schwartz,
obtiveram uma generaliza\c{c}\~{a}o estoc\'{a}stica das equa\c{c}\~{o}es de Hamilton \cite{CD:Ortega}.
Recentemente, num artigo de Cresson e Darses \cite{CD:CressonD1},
foi demonstrado que as equa\c{c}\~{o}es de Navier--Stokes admitem uma
estrutura Lagrangiana com a incorpora\c{c}\~{a}o de sistemas estoc\'{a}sticos Lagrangianos.
Estas equa\c{c}\~{o}es coincidem com as equa\c{c}\~{o}es de Euler--Lagrange
de uma funcional variacional estoc\'{a}stica, ou seja,
s\~{a}o sistemas de Euler--Lagrange estoc\'{a}sticos.

O Teorema de Noether estoc\'{a}stico \'{e} apresentado neste trabalho apenas para o caso aut\'{o}nomo.
Estabelecer um teorema de Noether estoc\'{a}stico para o caso n\~{a}o aut\'{o}nomo,
com mudan\c{c}a da vari\'{a}vel independente $t$, \'{e} um problema em aberto.



\end{document}